\documentclass[a4paper,12pt]{article}
\makeatletter
\newcommand*{\rom}[1]{\expandafter\@slowromancap\romannumeral #1@}
\makeatother

\usepackage{graphicx,amsfonts,amsmath,amsthm,amssymb,mathtools}
\usepackage[margin=0.84in]{geometry}
\usepackage{tikz-cd}
\usepackage{hyperref}

\newtheorem{theorem}{Theorem}[section]
\newtheorem{lemma}[theorem]{Lemma}
\newtheorem{corollary}[theorem]{Corollary}
\newtheorem{definition}{Definition}[section]
\newtheorem{proposition}[theorem]{Proposition}
\newtheorem{conjecture}{Conjecture}[section]

\begin{document}

\title{A Control Theorem for Primitive ideals\linebreak in Iwasawa algebras}

\author{Adam Jones}

\date{\today}

\maketitle

\begin{abstract}

\noindent Let $p$ be a prime, $K\backslash\mathbb{Q}_p$ a finite extension, $G$ a nilpotent, uniform pro-$p$ group. We prove that all faithful, primitive ideals in the Iwasawa algebra $KG$ are controlled by $C_G(Z_2(G))$, the centraliser of the second term in the upper central series for $G$.

\end{abstract}

\tableofcontents

\setcounter{section}{0}

\section{Introduction}

\subsection{Background}

Fix $p$ a prime, $K\backslash\mathbb{Q}_p$ a finite extension, $\mathcal{O}$ the ring of integers of $K$, $\pi$ a uniformiser for $\mathcal{O}$, and let $G$ be a uniform pro-$p$ group in the sense of \cite[Definition 4.1]{DDMS}.\\ 

\noindent Define the \emph{Iwasawa algebra} of $G$ with coefficients in $\mathcal{O}$ to be the $\mathcal{O}$-algebra defined by:

\begin{equation}
\mathcal{O}G:=\underset{N\trianglelefteq_o G}{\varprojlim} \mathcal{O}[G/N].
\end{equation}

\noindent An ongoing project is to investigate the ideal structure of $\mathcal{O}G$, and provide an answer to the questions posed in \cite[Section 7]{survey}. Much progress towards a complete classification of prime ideals was made in \cite{nilpotent} and \cite{APB}, but these results only apply in a characteristic $p$ setting, i.e. they apply only to prime ideals of $\mathcal{O}G$ which contain $p$. This paper will focus instead on the characteristic 0 case.\\

\noindent Define $KG:=\mathcal{O}G\otimes_{\mathcal{O}}K$, then using \cite[Lemma 21.1]{Schneider}, this can be realised as the algebra of continuous, $K$-valued distributions on $G$, i.e as a dense subalgebra of the distribution algebra $D(G,K)$ of $K$-analytic distributions on $G$ (see \cite{ST} for details). For this reason, we call $KG$ the \emph{Iwasawa algebra of continuous $K$-distributions} of $G$.

Note that $\mathcal{O}G$ is an $\mathcal{O}$-lattice subalgebra of $KG$ in the sense of \cite[Definition 2.7]{annals}, and that prime ideals of $KG$ are in bijection with prime ideals of $\mathcal{O}G$ that do not contain $\pi$. So our aim is to classify prime ideals in $KG$.\\

\noindent In the case where $G=\mathbb{G}(\mathcal{O})$ for some simply connected, split-semisimple, affine algebraic group $\mathbb{G}$, it is strongly believed that the only non-zero prime ideals of $KG$ have finite codimension (see \cite{munster} for more details). However, in this paper, we will be interested in the case when $G$ is solvable, where there is a much richer prime ideal structure.\\

\noindent The general conjecture is that prime ideals in $KG$ are all of a very predictable form, which is the essence of \cite[Question G]{survey}. This is made precise by the following definition:

\begin{definition}\label{standard}

Let $P$ be a prime ideal of $KG$. We say that $P$ is \emph{standard} if there exists a closed, normal subgroup $H$ of $G$ such that:\\

$i$. $H-1\subseteq P$.\\

$ii$. Setting $G':=\frac{G}{H}$, the ideal $\widetilde{P}:=P/(H-1)KG$ of $KG/(H-1)KG\cong KG'$ is centrally 

generated.

\end{definition}

\noindent Standard prime ideals have very nice properties, in particular they are generally \emph{completely prime}, i.e. if $P$ is standard then $KG/P$ is a domain.

\begin{conjecture}\label{reduction}

Let $G$ be a solvable group and let $P$ be a prime ideal of $KG$. Then $P$ is standard.

\end{conjecture}

\noindent The following definition (\cite[Definition 1.1]{APB}) will be essential throughout:

\begin{definition}

Let $I$ be a right ideal of $KG$:\\

\noindent 1. We say that $I$ is \emph{faithful} if for all $g\in G$, $g-1\in I$ if and only if $g=1$, i.e. $G\to KG/I,g\mapsto g+I$ is injective.\\

\noindent 2. We say that $H\leq_c G$ \emph{controls} $I$ if $I=(I\cap KH)KG$.\\

\noindent Define the \emph{controller subgroup} of $I$ by $I^{\chi}:=\bigcap\{U\leq_o G: U$ controls $I\}$, and denote by $Spec^f(KG)$ the set of all faithful prime ideals of $KG$.

\end{definition}

\noindent It follows from \cite[Theorem A]{controller} that a closed subgroup $H$ of $G$ controls $I$ if and only if $I^{\chi}\subseteq H$, and we see immediately from Definition \ref{standard} that a faithful prime ideal $P$ of $KG$ is standard if and only if it is controlled by the centre.\\

Therefore, to prove Conjecture \ref{reduction}, it should remain only to prove that for any faithful prime ideal $P$ of $KG$, $P$ is controlled by $Z(G)$.

\subsection{Main Results}

It was proved in \cite{nilpotent} that Conjecture \ref{reduction} holds in the analogous characteristic $p$ setting if we assume further that $G$ is \emph{nilpotent}:

\begin{definition}\label{upper-central}

Let $H$ be a group, $n\in\mathbb{N}$. Define the \emph{$n$'th centre} $Z_n(H)$ of $H$ inductively by $Z_0(H):=1$, and for $n>0$, $Z_n(H):=\{h\in H:(h,H)\subseteq Z_{i-1}(H)\}$.

Then each $Z_n(H)$ is a normal subgroup of $H$, $Z_1(H)=Z(H)$, and we define the \emph{upper central series} of $H$ to be the ascending chain of subgroups:

\begin{center}
$1=Z_0(H)\subseteq Z_1(H)\subseteq Z_2(H)\subseteq\cdots$
\end{center}

We say that $H$ is \emph{nilpotent} if the upper central series terminates at $H$, i.e. $Z_n(H)=H$ for some $n\in\mathbb{N}$. The \emph{nilpotency class} of $H$ is the smallest integer $n$ such that $Z_n(H)=H$.

\end{definition}

It is not difficult to see that any nilpotent group is solvable, and note that for each $i\geq 0$, $Z_i(G)$ is a closed, isolated normal subgroup of $G$.\\

\noindent Recall that an ideal $I$ of $KG$ is \emph{primitive} if $I=Ann_{KG}(M)$ for some irreducible $KG$-module $M$. Here is our main result:

\begin{theorem}\label{A}

Let $G$ be a nilpotent, uniform pro-$p$ group, and let $P$ be a faithful, primitive ideal of $KG$. Then $P$ is controlled by $C_G(Z_2(G))$, the centraliser of $Z_2(G)$ in $G$.

\end{theorem}

\noindent\textbf{Examples:} 1. If $G$ is nilpotent and contains a closed, isolated, abelian normal subgroup $H$ of codimension 1, it follows from Theorem \ref{A} that all faithful primitive ideals of $KG$ are controlled by $H$.\\

\noindent 2. Let $G$ be the group of unipotent, upper-triangular $4\times 4$ matrices over $\mathbb{Z}_p$, and let $\mathfrak{g}$ be the $\mathbb{Z}_p$-Lie algebra of $G$. 

Then $\mathfrak{g}$ is the Lie algebra of strictly upper triangular matrices over $p\mathbb{Z}_p$. This has basis $\{x_1,x_2,\cdots,x_6\}$ such that $x_3$ is central, $[x_1,x_4]=px_2$, $[x_1,x_5]=px_3$, $[x_2,x_6]=px_3$, $[x_4,x_6]=px_5$, all other brackets are zero.\\

\noindent Then $Z_2(\mathfrak{g})=$ Span$_{\mathbb{Z}_p}\{x_2,x_3,x_5\}$, and $C_{\mathfrak{g}}(Z_2(\mathfrak{g}))=$ Span$_{\mathbb{Z}_p}\{x_2,x_3,x_4,x_5\}$. This is abelian, so it follows from Theorem \ref{A} that any primitive ideal of $KG$ is controlled by an abelian subgroup of $G$.\\

\noindent 3. More generally, suppose that $G$ is the group of unipotent, upper triangular $n\times n$ matrices, then the $\mathbb{Z}_p$-Lie algebra $\mathfrak{g}$ of $G$ is the Lie algebra of strictly upper triangular $n\times n$ matrices over $p\mathbb{Z}_p$.

In this case, $C_{\mathfrak{g}}(Z_2(\mathfrak{g}))$ is precisely the set of all such matrices where the first two columns and the bottom two rows are zero, and it follows that $C_G(Z_2(G))$ is the group of unipotent, upper triangular matrices with zero in the (1,2)-position and in the (n-1,n)-position. For $n>4$, this is non-abelian.\\

\noindent Now, Theorem \ref{A} is not useful in the case when $Z(G)=1$, because in this case $C_G(Z_2(G))=G$ so the statement is obvious. Therefore, we draw no new conclusion if $G$ is semisimple, and in many cases when $G$ is solvable.\\

\noindent However, if $G$ is nilpotent, the result is highly useful, because for $G$ non-abelian, $C_G(Z_2(G))$ is always a proper subgroup of $G$. In particular, if $G$ has nilpotency class 2, then $Z_2(G)=G$ so $C_G(Z_2(G))=Z(G)$, and thus we achieve the following result:

\begin{corollary}\label{2-step}

Let $G$ be a uniform, nilpotent, pro-$p$ group of nilpotency class 2. Then all faithful, primitive ideals of $KG$ are standard.

\end{corollary}

\noindent To prove Theorem \ref{A}, we will adapt the approaches used in \cite{nilpotent}, \cite{controller} and \cite{APB} to study prime ideals in Iwasawa algebras in characteristic $p$:\\

\noindent Recall from \cite[\rom{2} 2.1.2]{Lazard} the definition of a $p$-valuation $\omega$ on a group $G$. We say that $G$ is \emph{$p$-valuable} if it carries a complete $p$-valuation and has finite rank, note that any closed subgroup of a uniform group is $p$-valuable.\\

\noindent The following result is analogous to \cite[Theorem B]{nilpotent}, and is the key step in the proof.

\begin{theorem}\label{B}

Let $G$ be a $p$-valuable group, and let $A$ be a closed, central subgroup of $G$. Let $P$ be a faithful, prime ideal of $\mathcal{O}G$ with $\pi\notin P$, and suppose there exists $\varphi\in Aut^{\omega}(G)$ such that $\varphi\neq 1$, $\varphi(P)=P$ and $\varphi(g)g^{-1}\in A$ for all $g\in G$.\\

\noindent Then if we assume that $\mathcal{O}A/\mathcal{O}A\cap P$ is finitely generated over $\mathcal{O}$, we conclude that $P$ is controlled by a proper, open subgroup of $G$.

\end{theorem}

\noindent We will prove Theorem \ref{B} in section 3, using the theory of Mahler expansions, and in section 4 we will show that this is enough to imply the appropriate control theorem for all prime ideals satisfying the appropriate finiteness condition, and ultimately prove Theorem \ref{A}.

\section{Preliminaries}

\subsection{The $C(G,\mathcal{O})$-action and Mahler expansions}

Fix $G$ a compact $p$-adic Lie group, and define: 

\begin{center}
$C(G,\mathcal{O}):=\{f:G\to\mathcal{O}: f$ continuous$\}$.
\end{center} 

\noindent Then $C(G,\mathcal{O})$ is an $\mathcal{O}$-algebra with pointwise addition and multiplication. Also, for each $n\in\mathbb{N}$, define: 

\begin{center}
$C_n=C(G,\frac{\mathcal{O}}{\pi^n\mathcal{O}}):=\{f:G\to\frac{\mathcal{O}}{\pi^n\mathcal{O}}:f$ locally constant$\}$.
\end{center}

\noindent Then each $C_n$ is an $\mathcal{O}$-algebra, and there exists a surjective map:

\begin{center}
$c_{n+1,n}:C_{n+1}\to C_n,f\mapsto (h:G\to\frac{\mathcal{O}}{\pi^n\mathcal{O}}, g\mapsto f(g)+\pi^n\mathcal{O})$.
\end{center} 

\noindent Furthermore, there exists a surjective map $c_n:C(G,\mathcal{O})\to C_n,f\mapsto (h:G\to \frac{\mathcal{O}}{\pi^n\mathcal{O}},g\mapsto f(g)+\pi^n\mathcal{O})$, and clearly $c_{n+1,n}\circ c_{n+1}=c_n$ for all $n$.

\begin{lemma}\label{limit}

$C(G,\mathcal{O})=\underset{n\in\mathbb{N}}{\varprojlim}$ $C_n$ in the category of $\mathcal{O}$-algebras.

\end{lemma}

\begin{proof}

Firstly, note that $C(G,\mathcal{O})$ is $\pi$-adically complete, and thus

\begin{equation}\label{cont}
C(G,\mathcal{O})=\underset{n\in\mathbb{N}}{\varprojlim}{C(G,\mathcal{O})/\pi^nC(G,\mathcal{O})}
\end{equation}

\noindent Consider the morphism of $\mathcal{O}$-algebras:

\begin{center}
$\Theta:C(G,\mathcal{O})\to C_n,f\mapsto (f':G\to \frac{\mathcal{O}}{\pi^n\mathcal{O}},g\mapsto f(g)+\pi^n\mathcal{O})$
\end{center}

\noindent It is clear that this map is surjective, and if $\Theta(f)=0$ then $f(g)\in\pi^n\mathcal{O}$ for all $g\in G$, so $f\in\pi^n C(G,\mathcal{O})$. Thus $ker(\Theta)=\pi^n C(G,\mathcal{O})$, so $C(G,\mathcal{O})/\pi^n C(G,\mathcal{O})\cong C_n$, and the result follows from (\ref{cont}).\end{proof}

\vspace{0.1in}

\noindent For convenience, set $A_n:=\frac{\mathcal{O}}{\pi^n\mathcal{O}}$ for each $n\in\mathbb{N}$, and note that $\mathcal{O}G/\pi^n\mathcal{O}G\cong A_nG$.\\ 

Recall from \cite[Proposition 2.5]{controller} that there exists an action $\rho_n:C_n\to End_{\mathcal{O}}(A_nG)$ for each $n$ such that if $U\leq_o G$ and $f\in C_n$ is constant on the cosets of $U$, and $r\in A_nG$ with $r=\underset{g\in G//U}{\sum}{s_gg}$ for some $s_g\in A_nU$, then $\rho_n(f)(r)=\underset{g\in G//U}{\sum}{f(g)s_gg}$.\\

\noindent Consider the canonical homomorphisms:

\begin{equation}
\nu_n:End_{\mathcal{O}}(A_nG)\to Hom(\mathcal{O}G,A_nG), f\mapsto (g:\mathcal{O}G\to\mathcal{O}G/\pi^n\mathcal{O}G,r\mapsto f(r+\pi^n\mathcal{O}G))
\end{equation}

\begin{equation}
\mu_{n+1,n}:Hom(\mathcal{O}G,A_{n+1}G)\to Hom(\mathcal{O}G,A_nG),f\mapsto (g:\mathcal{O}G\to\mathcal{O}G/\pi^n\mathcal{O}G,r\mapsto f(r)+\pi^n\mathcal{O}G)
\end{equation}

\noindent These give rise to the following commutative diagram for each $n\in\mathbb{N}$:

\begin{center}
\begin{tikzcd}

&C_n \arrow[r, "\rho_n"]   & End_{\mathcal{O}}(A_nG)\arrow[r,"\nu_n"] &Hom_{\mathcal{O}}(\mathcal{O}G,A_nG)\\
&C_{n+1} \arrow [u,"c_{n+1,n} "] \arrow[r, "\rho_{n+1}"] &End_{\mathcal{O}}(A_{n+1}G) \arrow[r,"\nu_{n+1}"]& Hom_{\mathcal{O}}(A_{n+1}G) \arrow[u,"\mu_{n+1,n}"]

\end{tikzcd}
\end{center}

\noindent Now, using a similar argument to Lemma \ref{limit}, we get $Hom_{\mathcal{O}}(\mathcal{O}G,\mathcal{O}G)=\underset{n\in\mathbb{N}}{\varprojlim}$ $Hom_{\mathcal{O}}(\mathcal{O}G,A_nG)$, so it follows that there is a unique map from $C(G,\mathcal{O})=\underset{n\in\mathbb{N}}{\varprojlim}$ $C_n$ to $Hom_{\mathcal{O}}(\mathcal{O}G,\mathcal{O}G)=End_{\mathcal{O}}{\mathcal{O}G}$ making the corresponding diagrams commute.

\begin{definition}\label{action}

Define $\rho:C(G,\mathcal{O})\to End_{\mathcal{O}}{\mathcal{O}G}$ to be the unique morphism defined above. We call this the \emph{canonical action} of $C(G,\mathcal{O})$ on $\mathcal{O}G$.

\end{definition}

\noindent Note that for each $n\in\mathbb{N}$, $f\in C_n$, $g\in G$, $\rho_n(f)(g)=f(g)g$, and it follows that for each $f\in C(G,\mathcal{O})$, we still have that $\rho(f)(g)=f(g)g$.

Also, note that if $f\in C(G,\mathcal{O})$ is locally constant, then $\rho(f)$ is the same as the endomorphism defined in \cite[Proposition 2.5]{controller}. Therefore we have the following result \cite[Proposition 2.8]{controller}, which will be useful to us later:

\begin{proposition}\label{loc-constant}

For each $U\leq_o G$, let $C_U:=\{f\in C(G,\mathcal{O}):f$ is constant on the cosets of $U\}$. Then $C_U$ is an $\mathcal{O}$-subalgebra of $C(G,\mathcal{O})$, and for any right ideal $I$ of $\mathcal{O}G$, $I$ is controlled by $U$ if and only if $\rho(C_U)(I)\subseteq I$.

\end{proposition}

\noindent Now we will recap the notion of the Mahler expansion of an automorphism of $G$. This was defined in \cite{nilpotent} working over a field of characteristic $p$, and we will now use the canonical action to extend the notion to fields of characteristic 0.\\

First recall the following result, first proved by Mahler in 1958, and given in full in \cite[\rom{3} 1.2.4]{Lazard}:

\begin{theorem}\label{Mahler}

Let $M$ be a complete topological $\mathbb{Z}_p$-module, and let $f:\mathbb{Z}_p^d\to M$ be a continuous function. Then for each $\alpha\in\mathbb{N}^d$, define $m_{\alpha}(f):=\underset{\beta\leq\alpha}{\sum}{(-1)^{\alpha-\beta}\binom{\alpha}{\beta}f(\beta)}$, where $\alpha\leq\beta$ if $\alpha_i\leq\beta_i$ for all $i$, and $\binom{\alpha}{\beta}=\underset{i\leq d}{\prod}{\binom{\alpha_i}{\beta_i}}$.

Then $m_{\alpha}(f)\rightarrow 0$ as $\vert\alpha\vert\rightarrow\infty$, and for all $\gamma\in\mathbb{Z}_p^d$, $f(\gamma)=\underset{\alpha\in\mathbb{N}^d}{\sum}{m_{\alpha}(f)\binom{\gamma}{\alpha}}$.

\end{theorem}

\noindent We call this $m_{\alpha}(f)$ the \emph{$\alpha$-Mahler coefficient} for $f$.\\

\noindent From now on, we will assume that $G$ is $p$-valuable, with $p$-valuation $\omega$, and we fix an ordered basis $\underline{g}=\{g_1,\cdots,g_d\}$ for $G$.\\ 

\noindent Then given $\varphi\in Aut^{\omega}(G):=\{\phi\in Aut(G):\omega(\phi(g)g^{-1})-\omega(g)>\frac{1}{p-1}$ for all $g\in G\}$, we have a continuous map $f:\mathbb{Z}_p^d\to\mathcal{O}G,\beta\mapsto\varphi(\underline{g}^{\beta})\underline{g}^{-\beta}$.\\

Using Theorem \ref{Mahler}, define $m_{\alpha}(\varphi):=m_{\alpha}(f)$ to be the \emph{$\alpha$-Mahler coefficient} for $\varphi$, and we have that for all $\beta\in \mathbb{Z}_p^d$:

\begin{center}
$\varphi(\underline{g}^{\beta})=\underset{\alpha\in\mathbb{N}^d}{\sum}{m_{\alpha}(\varphi)\binom{\beta}{\alpha}\underline{g}^{\beta}}$
\end{center}

\noindent Now, for each $\alpha\in\mathbb{Z}_p^d$, define $i_{\underline{g}}^{(\alpha)}:G\to\mathcal{O},\underline{g}^{\beta}\mapsto\binom{\beta}{\alpha}$, then clearly $i_{\underline{g}}^{(\alpha)}\in C(G,\mathcal{O})$, so let $\partial_{\underline{g}}^{(\alpha)}:=\rho(i_{\underline{g}}^{(\alpha)})\in End_{\mathcal{O}}(\mathcal{O}G)$. We call this the \emph{$\alpha$-quantized divided power}.\\

\noindent Then $\partial_{\underline{g}}^{(\alpha)}(g)=i_{\underline{g}}^{(\alpha)}(g)g$ for all $g\in G$, i.e. for all $\beta\in\mathbb{Z}_p^d$, $\partial_{\underline{g}}^{(\alpha)}(\underline{g}^{\beta})=\binom{\beta}{\alpha}\underline{g}^{\beta}$.\\

\noindent Hence $\varphi(g)=\underset{\alpha\in\mathbb{N}^d}{\sum}{m_{\alpha}(\varphi)\partial_{\underline{g}}^{(\alpha)}}(g)$ for all $g\in G$.\\

\noindent Therefore, after extending linearly to $\mathcal{O}[G]$ and passing to the completion, we should get (after establishing a suitable convergence condition for the Mahler coefficients) that:

\begin{equation}\label{Mahler1}
\varphi=\sum_{\alpha\in\mathbb{N}^d}{m_{\alpha}(\varphi)\partial_{\underline{g}}^{(\alpha)}}
\end{equation}

\noindent as an endomorphism of $\mathcal{O}G$. This is the \emph{Mahler expansion} of $\varphi$.\\

\noindent In section 3, we will see how to use Mahler expansions for suitable automorphisms to prove a control theorem.

\subsection{Non-commutative valuations}

Recall the definition of a \emph{filtration} $w:R\to\mathbb{Z}\cup\{\infty\}$ on a ring $R$ from \cite[Definition 2.2]{APB}, and of the subgroups $F_nR:=\{r\in R:w(r)\geq n\}$. Recall from \cite{LVO} that $w$ is a \emph{Zariskian filtration} if $F_1R\subseteq J(F_0R)$ and the \emph{Rees ring} $\underset{n\in\mathbb{N}}{\oplus}{F_nR}t^n$ is Noetherian.

Note that Zariskian filtrations are always separated.\\

\noindent An element $x\in R\backslash{0}$ is \emph{$w$-regular} if $w(xy)=w(x)+w(y)$ for all $y\in R$, and we say that $w$ is a valuation if every non-zero element of $R$ is $w$-regular. Finally, if $S$ is a central subring of $R$, we say that $w$ is $S$-linear if every non-zero element of $S$ is $w$-regular.\\

\noindent\textbf{Example:} If $\mathcal{O}$ is a complete, commutative DVR, and $(G,\omega)$ is a $p$-valuable group with ordered basis $\underline{g}=\{g_1,\cdots,g_d\}$, then $\mathcal{O}G=\{\underset{\alpha\in\mathbb{N}^d}{\sum}{\lambda_{\alpha}b_1^{\alpha_1}\cdots b_d^{\alpha_d}}:\lambda_{\alpha}\in\mathcal{O}\}$, where $b_i=g_i-1$.\\ 

\noindent By \cite[\rom{3} 2.3.3]{Lazard}, $\mathcal{O}G$ carries a complete, Zariskian valuation $w$ given by 

\begin{center}
$w(\underset{\alpha\in\mathbb{N}^d}{\sum}{\lambda_{\alpha}b_1^{\alpha_1}\cdots b_d^{\alpha_d}})=\inf\{v_p(\lambda_{\alpha})+\underset{i\leq n}{\sum}{\alpha_i\omega(g_i)}:\alpha\in\mathbb{N}^d\}$
\end{center}

\noindent and we may assume that the associated graded ring is a commutative polynomial ring over $k:=\mathcal{O}/\pi\mathcal{O}$ in $d+1$ variables.

\begin{lemma}\label{value}

Suppose $w$ is an $\mathcal{O}$-linear filtration on an $\mathcal{O}$-algebra $A$ of characteristic 0, and suppose that $x\in A$ with $w(x-1)>w(p)$. Then $w(x^{p^m}-1)=mw(p)+w(x-1)$ for all $m\in\mathbb{N}$.

\end{lemma}

\begin{proof}

Using the binomial theorem, it is clear that 

\begin{center}
$x^{p^m}-1=(x-1+1)^{p^m}-1=\underset{k\geq 1}{\sum}{\binom{p^m}{k}(x-1)^k}=p^m(x-1)+\underset{k\geq 2}{\sum}{\binom{p^m}{k}(x-1)^k}$
\end{center}

\noindent Clearly $w(p^m(x-1))=w(p^m)+w(x-1)=mw(p)+w(x-1)$ since $w$ is $\mathcal{O}$-linear. So it remains to show that $w(\binom{p^m}{k}(x-1)^k)>mw(p)+w(x-1)$ for all $k\geq 2$.\\

\noindent First, note that $w(\binom{p^m}{p^m}(x-1)^{p^m})=w((x-1)^{p^m})\geq p^mw(x-1)=(p^m-1)w(x-1)+w(x-1)>(p^m-1)w(p)+w(x-1)\geq mw(p)+w(x-1)$.

So from now on, we may assume that $k<p^m$.\\

\noindent Now, $k=a_0+a_1p+\cdots+a_tp^t$ for some integers $0\leq a_i<p$, and since $k\leq p^m-1$ we may assume that $t=m-1$. Define $s(k):=a_1+a_1+\cdots+a_t$.\\

\noindent Recall from \cite[\rom{3} 1.1.2.5]{Lazard} that $v_p(k!)=\frac{k-s(k)}{p-1}$, and hence\\

\noindent $v_p(\binom{p^m}{k})=v_p(p^m!)-v_p(k!)-v_p((p^m-k)!)=\frac{p^m-s(p^m)-k+s(k)-(p^m-k)+s(p^m-k)}{p-1}=\frac{s(k)+s(p^m-k)-s(p^m)}{p-1}$.\\

\noindent Now, let $m\geq i\geq 1$ be maximal such that $a_{m-i}\neq 0$, then $p^m=(p-1)p^{m-i}+(p-1)p^{m-i+1}+\cdots+(p-1)p^{m-1}+p^{m-i}$, so\\ 

\noindent $p^m-k=(p-1-a_{m-i})p^{m-i}+(p-1-a_{m-i+1})p^{m-i+1}+\cdots+(p-1-a_{m-1})p^{m-1}+p^{m-i}$\\

$=(p-a_{m-i})p^{m-i}+(p-1-a_{m-i+1})p^{m-i+1}+\cdots+(p-1-a_{m-1})p^{m-1}$.\\

Hence $s(p^m-k)=(p-a_{m-i})+(p-1-a_{m-i+1})+\cdots+(p-1-a_{m-1})=i(p-1)+1-s(k)$.\\

\noindent Clearly $s(p^m)=1$, so $v_p(\binom{p^m}{k})=\frac{s(k)+s(p^m-k)-s(p^m)}{p-1}=\frac{s(k)+i(p-1)+1-s(k)-1}{p-1}=\frac{i(p-1)}{p-1}=i$.\\

\noindent Also, $k=a_{m-i}p^{m-i}+\cdots+a_{m-1}p^{m-1}=p^{m-i}(a_{m-i}+\cdots+a_{m-1}p^{i-1})$, so $k\geq p^{m-i}\geq m-i+1$ if $i<m$.\\

\noindent Now, $w(\binom{p^m}{k}(x-1)^k)\geq w(\binom{p^m}{k})+kw(x-1)\geq v_p(\binom{p^m}{k})w(p)+kw(x-1)$\\

$=iw(p)+(k-1)w(x-1)+w(x-1)>iw(p)+(k-1)w(p)+w(x-1)$, so:\\

\noindent If $i<m$ then this is at least $iw(p)+(m-i)w(p)+w(x-1)=mw(p)+w(x-1)$, and\\

\noindent if $i=m$ then since $k>1$ we have $iw(p)+(k-1)w(p)+w(x-1)=mw(p)+(k-1)w(p)+w(x-1)\geq mw(p)+w(x-1)$ as required.\end{proof}

\vspace{0.1in}

Recall the following definition (\cite[Definition 3.1]{APB}):

\begin{definition}\label{non-com valuation}

A \emph{non-commutative valuation} on a simple artinian ring $Q$ is a Zariskian filtration $v:Q\to\mathbb{Z}\cup\{\infty\}$ such that if $\widehat{Q}$ is the completion of $R$ with respect to $v$, then $\widehat{Q}\cong M_k(Q(D))$ for some $k\in\mathbb{N}$ and some complete non-commutative DVR $D$, and $v$ is induced by the $J(D)$-adic filtration on $\widehat{Q}$.

\end{definition}

It follows from this definition that if $v$ is a non-commutative valuation on $Q$, then every element of $Z(Q)$ is $v$-regular.

\begin{theorem}\label{valuation}

Let $G$ be a $p$-valuable group, $P$ a faithful prime ideal of $\mathcal{O}G$ with $\pi\notin P$, then there exists a non-commutative valuation $v$ on $Q(\mathcal{O}G/P)$ such that the natural map $\tau:(\mathcal{O}G,w)\to (Q(\mathcal{O}G/P),v)$ is continuous, and there exists $m_1\in\mathbb{N}$ such that for all $g\in G$, $v(\tau(g^{p^{m_1}}-1))>v(p)$.

\end{theorem}

\begin{proof}

Since $\mathcal{O}G/P$ is a prime ring carrying a Zariskian filtration $w$, and the associated graded gr $\mathcal{O}G/P\cong $gr $\mathcal{O}G/gr P$ is a commutative, infinite dimensional $\mathbb{F}_p$-algebra, it follows from \cite[Theorem C]{nilpotent} that there exists a non-commutative valuation $v$ on $Q(\mathcal{O}G/P)$ such that $\tau$ is continuous.

Given $g\in G$, the sequence $g^{p^m}-1$ converges to zero in $\mathcal{O}G$, so since $\tau$ is continuous, the sequence $\tau(g^{p^m}-1)$ converges to zero with respect to $v$, and hence there exists $m_g\in\mathbb{N}$ such that $v(\tau(g^{p^{m_g}}-1))>v(p)$.\\

\noindent So, let $\underline{g}=\{g_1,\cdots,g_d\}$ be an ordered basis for $G$, and let $m_1:=\max\{m_{g_i}:i=1,\cdots,d\}$, then it follows that $v(\tau(g^{p^{m_1}}-1))>v(p)$ for all $g\in G$.\end{proof}

\section{Using Mahler expansions}

In section 2.1, we defined the Mahler expansion of an automorphism $\varphi\in Aut^{\omega}(G)$ as a linear combination in the quantized divided powers $\partial_{\underline{g}}^{\alpha}$. 

In this section, we will analyse convergence of Mahler expansions, and use this to prove our control Theorem \ref{B}.

\subsection{Mahler automorphisms}

We now define a class of automorphisms which have a workable formula for their Mahler coefficients.

\begin{definition}\label{Mahler-aut}

Suppose $\underline{g}=\{g_1,...,g_d\}$ is an ordered basis for $G$, and $\varphi\in Aut^{\omega}(G)$. We say that $\varphi$ is a \emph{Mahler automorphism with respect to $\underline{g}$} if its Mahler coefficients satisfy:\\

\noindent $m_{\alpha}(\varphi)=(\varphi(g_1)g_1^{-1}-1)^{\alpha_1}.....(\varphi(g_d)g_d^{-1}-1)^{\alpha_d}$ for all $\alpha\in\mathbb{Z}_p^d$.

\end{definition}

\noindent It follows from the proof of \cite[Corollary 6.6]{nilpotent} and \cite[Corollary 6.7]{nilpotent} that if $\varphi$ is a Mahler automorphism with respect to $\underline{g}$, then $\varphi=\underset{\alpha\in\mathbb{N}^d}{\sum}{m_{\alpha}(\varphi)}\partial_{\underline{g}}^{(\alpha)}$ as an endomorphism of $\mathcal{O}G$ as we require.\\

\noindent The following proposition simplifies the task of finding Mahler automorphisms. For convenience, we write $\psi(g):=\varphi(g)g^{-1}$ for $g\in G$:

\begin{proposition}\label{conditions}

The following are equivalent:\\

\noindent i. $\varphi$ is a Mahler automorphism with respect to $\underline{g}=\{g_1,\cdots,g_d\}$.

\noindent ii. $\psi(\underline{g}^{\beta})=\prod_{i=1}^{d}{\psi(g_i)^{\beta_i}}$ for all $\beta\in\mathbb{N}^d$.

\noindent iii. $\psi(g_i)$ commutes with $g_j$ for all $j\leq i$.

\end{proposition}

\begin{proof}

(ii$\implies$ i) Given by the proof of \cite[Lemma 6.7]{nilpotent}\\

\noindent (i$\implies$ ii) Suppose that for all $\alpha\in\mathbb{N}^d$:\\

\noindent $m_{\alpha}(\varphi)=\underset{\beta\leq\alpha}{\sum}(-1)^{\alpha-\beta}\binom{\alpha}{\beta}\psi(g_1^{\beta_1}g_2^{\beta_2}\cdots g_d^{\beta_d})=(\psi(g_1)-1)^{\alpha_1}(\psi(g_2)-1)^{\alpha_2}\cdots(\psi(g_d)-1)^{\alpha_d}$.\\

\noindent Then given $i=1,\cdots,d$, suppose that $\beta_j=0$ for all $j\neq i$. Clearly if $\beta_i=0,1$ then $\psi(g_i^{\beta_i})=\psi(g_i)^{\beta_i}$, so suppose, for induction that $\psi(g_i^{k})=\psi(g_i)^{k}$ for all $0\leq k<\beta_i$. Then:\\

\noindent $\underset{k\leq\beta_i}{\sum}{(-1)^{\beta_i-k}\binom{\beta_i}{k}\psi(g_i^{k})}=\underset{k<\beta_i}{\sum}{(-1)^{\beta_i-k}\binom{\beta_i}{k}\psi(g_i)^{k}+\psi(g_i^{\beta_i})}=(\psi(g_i)-1)^{\beta_i}$\\

\noindent $=\underset{k\leq\beta_i}{\sum}{\binom{\beta_i}{k}(-1)^{\beta_i-k}\psi(g_i)^k}$ using binomial expansion, hence $\psi(g_i^{\beta_i})=\psi(g_i)^{\beta_i}$.\\

\noindent So, we have shown that $\psi(g_i^{\beta_i})=\psi(g_i)^{\beta_i}$ for all $i=1,\cdots,d$, $\beta_i\in\mathbb{Z}_p$.\\

Now suppose that for $i=1,\cdots,m$, $\psi(g_1^{\beta_1}\cdots g_m^{\beta_m})=\psi(g_1)^{\beta_1}\cdots\psi(g_m)^{\beta_m}$ for some $1\leq m<d$. We have proved this for $m=1$, so we will now show that it holds for $m+1$ and apply induction.\\

So now suppose that for all $k<\beta_{m+1}$, $\psi(g_1^{\beta_1}\cdots g_m^{\beta_m}g_{m+1}^k)=\psi(g_1^{\beta_1}\cdots g_m^{\beta_m})\psi(g_{m+1})^k$. It is clear that this holds for $\beta_{m+1}=1$, so we will apply a second induction.\\

\noindent We know that $\underset{\gamma\leq\beta}{\sum}{(-1)^{\beta-\gamma}\binom{\beta}{\gamma}\psi(g_1^{\gamma_1}\cdots g_{m+1}^{\gamma_{m+1}})}$

\noindent $=\sum_{k=0}^{\beta_{m+1}}\underset{\gamma\leq\beta}{\sum}{(-1)^{\beta-\gamma}\binom{\beta}{\gamma}(-1)^{\beta_{m+1}-k}\binom{\beta_{m+1}}{k}\psi(g_1^{\gamma_1}\cdots g_m^{\gamma_m}g_{m+1}^k)}$

\noindent $=(\psi(g_1)-1)^{\beta_1}(\psi(g_2)-1)^{\beta_2}\cdots \psi(g_{m+1})-1)^{\beta_{m+1}}$.\\

\noindent So expanding out $(\psi(g_{m+1})-1)^{\beta_{m+1}}$ and applying the first induction gives that:\\

\noindent $\underset{\gamma\leq\beta}{\sum}{(-1)^{\beta-\gamma}\binom{\beta}{\gamma}\psi(g_1)^{\gamma_1}\cdots (g_m)^{\gamma_m}\sum_{k=0}^{\beta_{m+1}}(-1)^{\beta_{m+1}-k}\binom{\beta_{m+1}}{k}\psi(g_{m+1})^k}$\\

\noindent $=\sum_{k=0}^{\beta_{m+1}}\underset{\gamma\leq\beta}{\sum}{(-1)^{\beta-\gamma}(-1)^{\beta_{m+1}-k}\binom{\beta}{\gamma}\binom{\beta_{m+1}}{k}\psi(g_1^{\gamma_1}\cdots g_m^{\gamma_m}g_{m+1}^k)}$.\\

\noindent So applying the second induction gives $\psi(g_1^{\beta_1}\cdots g_{m+1}^{\beta_{m+1}})=\psi(g_1)^{\beta_1}\cdots \psi(g_{m+1})^{\beta_{m+1}}$, so we are done.\\

\noindent (ii$\implies$ iii) Given $i\leq j$, we have that $\psi(g_ig_j)=\psi(g_i)\psi(g_j)$, so by the definition of $\psi$, $\varphi(g_ig_j)(g_ig_j)^{-1}=\varphi(g_i)g_i^{-1}\varphi(g_j)g_j^{-1}$.\\

\noindent Therefore, $\varphi(g_i)\varphi(g_j)g_j^{-1}g_i^{-1}=\varphi(g_i)g_i^{-1}\varphi(g_j)g_j^{-1}$, so $\varphi(g_j)g_j^{-1}g_i^{-1}=g_i^{-1}\varphi(g_j)g_j^{-1}$, i.e. $g_i^{-1}\psi(g_j)=\psi(g_j)g_i^{-1}$.\\

\noindent So $\psi(g_j)$ commutes with $g_i^{-1}$, and hence it commutes with $g_i$.\\

\noindent (iii$\implies$ ii) This is clear from the definition of $\psi$.\end{proof}

\noindent\textbf{\underline{Examples:}} 1. Suppose that $\varphi\in Aut^{\omega}Z(G)$ is trivial mod centre, i.e. $\varphi(g)g^{-1}\in Z(G)$ for all $g\in G$. Then for any ordered basis $\underline{g}=\{g_1,\cdots,g_d\}$, $\varphi(g_i)g_i^{-1}$ is central, and so commutes with $g_1,\cdots,g_i$.\\

\noindent So by the proposition, $\varphi$ is a Mahler automorphism with respect to any basis. In fact, the proposition shows that if $\varphi$ is a Mahler automorphism with respect to any basis, then $\varphi$ is trivial mod centre. But it is possible for $\varphi$ to be a Mahler automorphism with respect to \emph{some} basis, and yet not be trivial mod centre.\\

\noindent 2. Recall from \cite{APB} that $G$ is \emph{abelian-by-procyclic} if $G\cong\mathbb{Z}_p^d\rtimes\mathbb{Z}_p$. Then if we take $\varphi\in Inn(G)$ to be the automorphism defining the action of $\mathbb{Z}_p$ on $\mathbb{Z}_p^d$, then $\varphi$ is a Mahler automorphism with respect to any basis of the form $\{h_1,\cdots,h_d,X\}$, where $\{h_1,\cdots,h_d\}$ is a basis for $\mathbb{Z}_p^d$.

Note that in this case, if $G$ is not nilpotent, or has nilpotency class greater than 2, then $\varphi$ will not be trivial mod centre.\\

\noindent Since we want $\varphi$ to be a Mahler automorphism with respect to any basis, we will assume from now on that $\varphi$ is trivial mod centre. However, it should be noted that to prove general control theorems for Iwasawa algebras in characteristic $p$ for non-nilpotent groups, the theory of general Mahler automorphisms is highly useful, as demonstrated by the approach in \cite{APB}, so this theory should be developed further.

\subsection{Approximations}

From now on, we will fix $P$ a faithful, prime ideal of $\mathcal{O}G$ such that $\pi\notin P$, $\varphi\in Aut^{\omega}(G)$ an automorphism, trivial mod centre, $\varphi\neq id$, and we assume that $\varphi(P)=P$.

Using Theorem \ref{valuation}, we fix a non-commutative valuation $v$ on $Q(\mathcal{O}G/P)$, and let $\tau:\mathcal{O}G\to Q(\mathcal{O}G/P)$ be the natural map, which is continuous.\\

\noindent Recall from \cite[Proposition 4.9]{nilpotent} that if we define $z(\varphi):G\to Z(G),g\mapsto\underset{m\rightarrow\infty}{\lim}{(\varphi^{p^m}(g)g^{-1})^{p^{-m}}}$, then $z(\varphi)$ is a group homomorphism, $z(\varphi^{p^m})(g)=z(\varphi)(g)^{p^m}$ for all $g\in G$, and:

\begin{center}
$\varphi^{p^m}(g)g^{-1}\equiv z(\varphi)(g)^{p^m}$ ($mod$ $Z(G)^{p^{2m}}$) for all $m\in\mathbb{N}$.
\end{center}

\noindent Using Theorem \ref{valuation}, choose an integer $m_1>0$ such that for all $Z\in Z(G)$, $v(\tau(Z^{p^{m_1}}-1))>v(p)$, and from now on, fix $z:=z(\varphi^{p^{m_1}}):G\to Z(G)$.\\ 

\noindent Set $\lambda:=\inf\{v(\tau(z(g)-1):g\in G\}$, then since $P$ is faithful and $\varphi\neq id$, the proof of \cite[Lemma 7.5]{nilpotent} shows that $\lambda<\infty$. Furthermore, since $z(g)=z(\varphi^{p^{m_1}})(g)=z(\varphi)(g)^{p^{m_1}}$, it follows from our choice of $m_1$ that $\lambda>v(p)$.\\

\noindent Now, for any ordered basis $\underline{g}=\{g_1,\cdots,g_d\}$ for $G$, set $q_{i,m}:=\tau(z(g_i)^{p^m}-1)$ for each $i\leq d$, $m\in\mathbb{N}$. We write $\underline{q}_m^{\alpha}:=q_{1,m}^{\alpha_1}q_{2,m}^{\alpha_2}\cdots q_{d,m}^{\alpha_d}$ as usual, and it is clear that $v(q_{i,0})=\lambda$ for some $i$.

Moreover, if $v(q_{i,0})=\lambda$ then $v(q_{i,m})=v(\tau(z(g_i)^{p^m}-1))=\lambda+mv(p)$ by Lemma \ref{value}.\\

\noindent The following result is analogous to \cite[Proposition 7.7]{nilpotent}:

\begin{proposition}\label{approx}

For any $\alpha\in\mathbb{N}^d$ with $\alpha\neq 0$, we have that for sufficiently high $m\in\mathbb{N}$:

\begin{center}
$v(\tau(m_{\alpha}(\varphi^{p^m}))-\underline{q}_m^{\alpha})\geq (2m-m_1)v(p)+(\vert\alpha\vert-1)\lambda$
\end{center}

Where $\vert\alpha\vert:=\alpha_1+\cdots+\alpha_d$.

\end{proposition}

\begin{proof}

We have that $m_{\alpha}(\varphi^{p^m})=(\varphi^{p^m}(g_1)g_1^{-1}-1)^{\alpha_1}\cdots(\varphi^{p^m}(g_d)g_d^{-1}-1)^{\alpha_d}$ since $\varphi^{p^m}$ is a Mahler automorphism, so let $\delta_{i,m}:=\tau(\varphi^{p^m}(g_i)g_i^{-1}-1)-q_{i,m}$, then:

\begin{center}
$\tau(m_{\alpha}(\varphi^{p^m}))=(\delta_{1,m}+q_{1,m})^{\alpha_1}\cdots (\delta_{d,m}+q_{d,m})^{\alpha_d}$
\end{center}

\noindent Hence $m_{\alpha}(\varphi^{p^m})-\underline{q}_m$ is a sum of words in $q_{i,m}$ and $\delta_{i,m}$ of length $\vert\alpha\vert$, each containing at least one $\delta_{i,m}$.\\

\noindent So, it remains to show that $v(\delta_{i,m})\geq (2m-m_1)v(p)$ for each $i,m$, and the result follows.\\

\noindent Since $\varphi^{p^{m+m_1}}(g_i)g_i^{-1}\equiv z(g_i)^{p^{m}}$ ($mod$ $Z(G)^{p^{2(m+m_1)}}$), if we assume that $m\geq m_1$ we have that $\varphi^{p^m}(g_i)g_i^{-1}=z(g_i)^{p^{m-m_1}}\epsilon_{i,m}^{p^{2m}}$ for some $\epsilon_{i,m}\in Z(G)$. Hence:

\begin{center}
$\delta_{i,m}=\tau(z(g_i)^{p^{m-m_1}}\epsilon_{i,m}^{p^{2m}}-1)-\tau(z(g_i)^{p^{m-m_1}}-1)=\tau(z(g_i)^{p^{m-m_1}})\tau(\epsilon_{i,m}^{p^{2m}}-1)$
\end{center}

Since $\epsilon_{i,m}\in Z(G)$ and $m>m_1$, we have that $v(\tau(\epsilon_{i,m}^{p^{m_1}}-1))\geq v(p)$, so by Lemma \ref{value}:

\begin{center}
$v(\tau(\epsilon_{i,m}^{p^{2m}}-1))=(2m-m_1)v(p)+v(\tau(\epsilon_{i,m}^{p^{m_1}}-1))$
\end{center}

\noindent Hence $v(\delta_{i,m})\geq (2m-m_1)v(p)$ as required.\end{proof}

\vspace{0.2in}

\noindent Now, we know that $\tau\varphi^{p^m}=\underset{\alpha\in\mathbb{N}^d}{\sum}{\tau(m_{\alpha}(\varphi^{p^m}))\tau\partial_{\underline{g}}^{(\alpha)}}$, so using the proposition, we set $\varepsilon_m:=\underset{\alpha\neq 0}{\sum}{(\tau(m_{\alpha}(\varphi^{p^m}))-\underline{q}_m^{\alpha}})\tau\partial_{\underline{g}}^{(\alpha)}$, and we have that $v(\varepsilon_m(y))\geq (2m-m_1)v(p)$ for all $y\in\mathcal{O}G$, and:

\begin{equation}\label{Mahler2}
\tau\varphi^{p^m}-\tau=\underset{\alpha\neq 0}{\sum}{\underline{q}_m^{\alpha}\tau\partial_{\underline{g}}^{(\alpha)}}+\varepsilon_m
\end{equation}

\noindent Now, for each $i=1,\cdots,d$, set $\partial_i:=\partial_{\underline{g}}^{(e_i)}$, where $e_i$ is the standard $i$'th basis vector. Then since $\varphi(P)=P$, we have that for any $y\in P$, $(\tau\varphi^{p^m}-\tau)(y)=0$, so:

\begin{equation}\label{Mahler3}
0=q_{1,m}\tau\partial_1(y)+\cdots+q_{d,m}\tau\partial_{d}(y)+\underset{\vert\alpha\vert\geq 2}{\sum}{\underline{q}_m^{\alpha}\tau\partial_{\underline{g}}^{(\alpha)}(y)}+\varepsilon_m(y)
\end{equation}

\noindent We want to analyse convergence of this expression, from which we hope to deduce a control theorem.

\subsection{Drawbacks in characteristic 0}

So far, these techniques are identical to the methods oulined and used in \cite{nilpotent} and \cite{APB}, but in both these cases, the underlying ring had characteristic $p$. When working in characteristic zero, potentially fatal problems arise.\\

\noindent The key issue is the linear growth rate of the Mahler approximations $q_{i,m}=\tau(z(g_i)^{p^m}-1)$, as described in Lemma \ref{value}, which arises due to our assumption that $\pi\notin P$. Suppose we assume instead that $\pi\in P$, and hence $Q(\mathcal{O}G/P)$ has characteristic $p$. In this case: 

\begin{center}
$q_{i,m}:=\tau(z(g_i)^{p^m}-1)=\tau(z(g_i)-1)^{p^m}$ for each $i$.
\end{center}

\noindent Therefore, $v(q_{i,m})=p^m\lambda$, and hence $v(\underline{q}_m^{\alpha})\geq p^m\vert\alpha\vert\lambda$, so the Mahler approximations grow exponentially with $m$. In this case we may divide out our expression (\ref{Mahler3}) by an expression involving only the lower order terms $q_{i,m}$, and this will not affect the convergence of the higher order terms.

Using this approach, we can deduce that $\tau\partial_i(y)=0$ for some $i$, and a control theorem will follow. This is outlined in full detail in \cite[Section 7]{nilpotent}.\\

\noindent But in our case, we have that $\pi\notin P$, and hence $Q(\mathcal{O}G/P)$ has characteristic 0, and we see that $v(q_{m}^{\alpha})\geq m\vert\alpha\vert v(p)+\vert\alpha\vert\lambda$ using Lemma \ref{value}, and equality holds for for certain values of $\alpha$.

This means that the Mahler approximations grow linearly with $m$, and it is not possible to remove all the lower order terms in expression (\ref{Mahler3}) without affecting convergence of the higher order terms.\\

\noindent We will now explore particular instances when we can deduce a convergence argument. These instances occur when we impose restrictions on the ideal $P$ to ensure that the centre is as small as possible, and at present they seem like the strongest we can obtain from Mahler expansions in characteristic zero.

\subsection{Using Compactness}

From now on, we will assume, as in the statement of Theorem \ref{B}, that we can find a closed, isolated subgroup $A$ of $Z(G)$ such that $\mathcal{O}A/\mathcal{O}A\cap P$ is finitely generated over $\mathcal{O}$ for some $n\in\mathbb{N}$, and we will suppose that $\varphi(g)g^{-1}\in A$ for all $g\in G$.\\

\noindent Then since $\varphi^{p^{m+m_1}}(g)g^{-1}\in A$ for all $g\in G$ and $A$ is isolated, it follows that $(\varphi^{p^{m+m_1}}(g)g^{-1})^{p^{-m}}\in A$ for all $m$. So since $A$ is closed, $z(g)=\underset{m\rightarrow\infty}{\lim}{(\varphi^{p^{m+m_1}}(g)g^{-1})^{p^{-m}}}\in A$, and thus for all $m\in\mathbb{N}$:

\begin{center}
$q_{i,m}=\tau(z(g_i)^{p^m}-1)\in \mathcal{O}A/\mathcal{O}A\cap P$.
\end{center}

\noindent Now, fix $R:=\mathcal{O}A/\mathcal{O}A\cap P$, then $R$ is a subring of the $K$-algebra $Q(\mathcal{O}G/P)$. So let $F$ be the $K$-span of $R$, then since $R$ is a domain, finitely generated over $\mathcal{O}$, $F$ is a finite field extension of $K$, and clearly $F\subseteq Z(Q(\mathcal{O}G/P))$.

So, let $\mathcal{V}$ be the valuation ring for $F$, and let $\mu\in\mathcal{V}$ be a uniformiser. Then since $v|_F$ is a valuation on $F$, it follows that $v|_F$ is a scalar multiple of the standard $\mu$-adic valuation on $F$, and hence $\mathcal{V}=\{\beta\in F:v(\beta)>0\}$.\\

\noindent Now, choose $i$ such that $q_{i,0}$ has value $\lambda$, we may assume without loss of generality that $i=1$. Then since $\lambda>v(p)$, $v(q_{1,m})=mv(p)+\lambda$ for all $m$ by Lemma \ref{value}.\\

\noindent Now, the key property of $\mathcal{V}$ which we can exploit is compactness, which implies that any sequence in $\mathcal{V}$ has a convergent subsequence. Using this notion, we obtain the following result:

\begin{proposition}\label{convergence}

For any ordered basis $\underline{g}=\{g_1,\cdots,g_d\}$, set $q_{i,m}:=\tau(z(g_i)^{p^m}-1)$ with $v(q_{1,0})=\lambda$. Then for each $i=1,\cdots,d$, the sequence $q_{1,m}^{-1}q_{i,m}$ has a subsequence converging to some $\beta_i\in\mathcal{V}$ as $m\rightarrow\infty$, and $(\beta_1\tau\partial_1+\cdots+\beta_d\tau\partial_d)(P)=0$.

\end{proposition}

\begin{proof}

For each $m$, $q_{1,m}$ is $v$-regular of value $\lambda+mv(p)$, so since $v(q_{i,m})\geq\lambda+mv(p)$, it follows that $v(q_{1,m}^{-1}q_{i,m})\geq\lambda+mv(p)-\lambda-mv(p)=0$, so $q_{1,m}^{-1}q_{i,m}\in\mathcal{V}$.\\ 

Therefore, using compactness of $\mathcal{V}$, we can choose a subsequence $\underline{a}:=(m_1,m_2,\cdots)$ with $m_1<m_2<m_3<\cdots$ such that $q_{1,m_j}^{-1}q_{i,m_j}$ converges in $\mathcal{V}$ as $j\rightarrow\infty$.\\ 

\noindent Let $\beta_i:=\underset{j\rightarrow\infty}{\lim}{q_{1,m_j}^{-1}q_{i,m_j}}\in\mathcal{V}$.\\

\noindent Now, given $\alpha\in\mathbb{N}^d$ with $\vert\alpha\vert\geq 2$; $v(q_{1,m}^{-1}\underline{q}_m^{\alpha})=v(\underline{q}_m^{\alpha})-v(q_{1,m})\geq m\vert\alpha\vert v(p)+\vert\alpha\vert\lambda-mv(p)-\lambda\geq mv(p)+\lambda\rightarrow\infty$ as $m\rightarrow\infty$.

Also, $v(q_{1,m}^{-1}\varepsilon_m)=v(\varepsilon_m)-v(q_{1,m})\geq (2m-m_1)v(p)-mv(p)-\lambda=(m-m_1)v(p)-\lambda\rightarrow\infty$ as $m\rightarrow\infty$.\\

\noindent Therefore, $q_{1,m}^{-1}\underline{q}_{m}^{\alpha}\rightarrow 0$ as $m\rightarrow\infty$, $q_{1,m}^{-1}\varepsilon_m\rightarrow 0$ as $m\rightarrow\infty$. So, dividing out our expression (\ref{Mahler3}) by $q_{1,m}$, we obtain:

\begin{equation}\label{Mahler4}
0=q_{1,m}^{-1}q_{1,m}\tau\partial_1(y)+q_{1,m}^{-1}q_{2,m}\tau\partial_2(y)+\cdots+q_{1,m}^{-1}q_{d,m}\tau\partial_d(y)+\underset{\vert\alpha\vert\geq 2}{\sum}{q_{1,m}^{-1}\underline{q}_m^{\alpha}\tau\partial_{\underline{g}}^{(\alpha)}(y)}+q_{1,m}^{-1}\varepsilon_m(y)
\end{equation}

\noindent So considering the subsequence associated with $\underline{a}=(m_1,m_2,\cdots)$, and taking the limit, we get that $\beta_1\tau\partial_1(y)+\beta_2\tau\partial_2(y)+\cdots+\beta_d\tau\partial_d(y)=0$.\\

\noindent Since our choice of $y\in P$ was arbitrary, it follows that $(\beta_1\tau\partial_1+\cdots+\tau\partial_d)(P)=0$ as required.\end{proof}

\noindent\textbf{\underline{Note:}} These elements $\beta_i$ depend on the choice of ordered basis $\underline{g}$..\\

\noindent From now on, fix an ordered basis $\underline{g}=\{g_1,\cdots,g_d\}$, and fix $\beta_1,\cdots,\beta_d$ as in the statement of Proposition \ref{convergence}, and define:

\begin{center}
$h_{\underline{g}}:\mathcal{O}G\to Q(\mathcal{O}G/P),x\mapsto (\beta_1\tau\partial_1+\beta_2\tau\partial_2+\cdots+\beta_d\tau\partial_d)(x)$
\end{center} 

\noindent Then $h_{\underline{g}}\in Hom_{\mathcal{O}}(\mathcal{O}G,Q(\mathcal{O}G/P))$, and it follows from Proposition \ref{convergence} that $h_{\underline{g}}(P)=0$, and thus $h_{\underline{g}}\in Hom_{\mathcal{O}}(\mathcal{O}G/P,Q(\mathcal{O}G/P))$.\\

\noindent Also, since each $\beta_i$ lies in $F=(\mathcal{O}A/\mathcal{O}A\cap P)\otimes_{\mathcal{O}}K$, it follows that the image of $\mathcal{O}G/P$ under $h_{\underline{g}}$ lies in $(\mathcal{O}G/P)\otimes_{\mathcal{O}}K=KG/(P\otimes_{\mathcal{O}}K)$.

So since it is clear that each $\partial_i$ extends to a $K$-linear endomorphism of $KG$, we may assume that $h_{\underline{g}}$ lies in $Hom_{\mathcal{O}}(KG/(P\otimes_{\mathcal{O}}K),KG/(P\otimes_{\mathcal{O}}K))$, and hence it makes sense to raise $h_{\underline{g}}$ to integer powers.\\

\noindent We now need some technical results:

\begin{lemma}\label{valuation-ring}

Fix $R=\frac{\mathcal{O}A}{\mathcal{O}A\cap P}$, then there exists $s\in\mathbb{N}$ such that $\pi^s\mathcal{V}\subseteq R$. Also, there exists $t\in\mathbb{N}$ such that if $x\in\mathcal{V}$ and $v(x)>0$ then $x^t\in \pi R$.

\end{lemma}

\begin{proof}

Since $R$ is a lattice in $F$, which is a finite dimensional $\mathbb{Q}_p$-vector space, it follows that $R$ is a free $\mathbb{Z}_p$-module of rank dim$_{\mathbb{Q}_p}F$. This is also the rank of $\mathcal{V}$, and it follows that $R$ has finite index in $\mathcal{V}$, and hence $p^l\mathcal{V}\subseteq R$ for some $l\in\mathbb{N}$, and the first statement follows.\\

\noindent Now, if $v(x)>0$, then $x=\pi^{-k}r$ for some $r\in R$, $k\in\mathbb{N}$, and $v(r)-kv(\pi)\geq 1$. We know that there exists $s\in\mathbb{N}$ such that $\pi^s\mathcal{V}\subseteq R$, so choose $t\in\mathbb{N}$ such that $t\geq (s+1)v(\pi)$.\\

Then $x^t=\pi^{-kt}r^t=\pi^{s+1}(\pi^{-(kt+s+1)}r^t)$, and note that: 

\begin{center}
$v(\pi^{-(kt+s+1)}r^t)=t(v(r)-kv(\pi))-(s+1)v(\pi)\geq t-(s+1)v(\pi)\geq 0$
\end{center}

\noindent Thus $\pi^{-(kt+s+1)}r^t\in\mathcal{V}$, so since $\pi^{s+1}\mathcal{V}\subseteq \pi R$ and $x^t=\pi^{s+1}(\pi^{-(kt+s+1)}r^t)$, it follows that $x^t\in \pi R$ as required.\end{proof}

\noindent Now, recall from the definition of $\partial_1,\cdots,\partial_d$, that if $g=\underline{g}^{\alpha}\in G$ for some $\alpha\in\mathbb{Z}_p^d$, then $\partial_i(g)=\alpha_i g$. So if we let $k_{g,\underline{g}}:=\beta_1\alpha_1+\beta_2\alpha_2+\cdots+\beta_d\alpha_d\in\mathcal{V}$, then $h_{\underline{g}}(g)=k_{g,\underline{g}} g$ for all $g\in G$.\\

\begin{lemma}\label{derivation}

For every ordered basis $\underline{g}$ of $G$, $h_{\underline{g}}$ is $F$-linear.

\end{lemma}

\begin{proof}

Since $h_{\underline{g}}(a)=k_{a,\underline{g}}a$ for every $a\in A$, and $k_{a,\underline{g}}\in F$, it follows that $h_{\underline{g}}$ sends $\mathcal{O}A$ to $F=\mathcal{O}A/(\mathcal{O}A\cap P)\otimes_{\mathcal{O}}K$, and hence it sends $F$ to $F$, i.e. $h_{\underline{g}}$ restricts to a $K$-linear endomorphism of $F$.\\

Also, for every $i$, $\partial_i$ is a $K$-linear derivation of $\mathcal{O}G$. So since each $\beta_i$ is central, it follows that $h_{\underline{g}}=\beta_1\tau\partial_1+\cdots+\beta_d\tau\partial_d$ is also a $K$-linear derivation, so $h_{\underline{g}}$ restricts to a $K$-linear derivation of $F$.\\

\noindent Using the derivation property, we see that to prove that $h_{\underline{g}}$ is $F$-linear, it remains only to prove that $h_{\underline{g}}(F)=0$. We will show, in fact, that all $K$-linear derivations of $F$ are zero.\\

\noindent Suppose that $\alpha\in F$, then since $F$ is a finite extension of $K$, $\alpha$ is the root of some polynomial $f(x)=a_0+a_1x+\cdots+a_nx^n$ with coefficients in $K$, and we will assume that $n$ is minimal. Let $\delta$ be a $K$-linear derivation of $F$, then:\\

\noindent $0=\delta(0)=\delta(a_0+a_1\alpha+\cdots+a_n\alpha^n)=a_1\delta(\alpha)+\cdots+a_n\delta(\alpha^n)$\\

$=a_1\delta(\alpha)+\cdots+na_n\alpha^{n-1}\delta(\alpha)=\delta(\alpha)(a_1+2a_2\alpha+\cdots+na_n\alpha^{n-1})$.\\

\noindent So if $\delta(\alpha)\neq 0$ then $a_1+2a_2\alpha+\cdots+na_n\alpha^{n-1}=0$, contradicting minimality of $n$. Hence $\delta(\alpha)=0$, meaning that $\delta=0$.\end{proof}

\noindent Now, let $f$ be the degree of the residue field of $F$, then for all $\beta\in\mathcal{V}$: 

\begin{equation}
\beta^{p^f-1}\equiv \begin{cases} 
      1 &  v(\beta)=0 \\
      0 &  v(\beta)>0 \\ 
   \end{cases}
 (\text{mod } \mu\mathcal{V})
\end{equation}

\noindent It follows easily that for all $n\in\mathbb{N}$:

\begin{equation}\label{little}
\beta^{p^n(p^f-1)}\equiv \begin{cases} 
      1 &  v(\beta)=0 \\
      0 &  v(\beta)>0 \\ 
   \end{cases}
 (\text{mod } \mu^{n+1}\mathcal{V})
\end{equation}

\noindent Using Lemma \ref{valuation-ring}, choose an integer $t_0\geq 0$ such that $\mu^{t_0+1}\mathcal{V}\subseteq \pi R$ where $R=\mathcal{O}A/\mathcal{O}A\cap P$ and $x^{p^{t_0}}\in \pi R$ for all $x\in\mathcal{V}$ with $v(x)>0$. Then examining (\ref{little}) shows that $\beta^{p^{t_0}(p^f-1)}\in R\subseteq\mathcal{O}G/P$ for all $\beta\in\mathcal{V}$.\\

\noindent Also, using Lemma \ref{derivation}, we see that $h_{\underline{g}}$ is $F$-linear. So since $k_{\underline{g},g}\in F$, it follows that $h_{\underline{g}}^n(g)=k_{\underline{g},g}^ng$ for all $g\in G$, $n\in\mathbb{N}$.\\

Therefore, if we define $w_{\underline{g}}:=h_{\underline{g}}^{p^{t_0}(p^f-1)}$, then for all $g\in G$, $n\in\mathbb{N}$, since $k_{g,\underline{g}}^{p^{n+t_0}(p^f-1)}\in\mathcal{O}G/P$, we have that $w_{\underline{g}}^{p^n}(g)=k_{g,\underline{g}}^{p^{n+t_0}(p^f-1)}g\in\mathcal{O}G/P$, and thus $w_{\underline{g}}$ sends $\frac{\mathcal{O}G}{P}$ to $\frac{\mathcal{O}G}{P}$.\\ 

\noindent Fixing $g\in G$, $n\in\mathbb{N}$, using (\ref{little}) we have:

\begin{center}
$w_{\underline{g}}^{p^n}(g)\equiv \begin{cases} 
      1 &  v(k_{g,\underline{g}})=0 \\
      0 &  v(k_{g,\underline{g}})>0 \\ 
   \end{cases}$
 ($mod$ $\pi^{n+1}\frac{\mathcal{O}G}{P}$)
\end{center}

\noindent So, setting $S:=\frac{\mathcal{O}G}{P}$ for convenience, consider the composition:

\begin{center}
\begin{tikzcd}

&\mathcal{O}G \arrow[r, "\tau"]   & S \arrow[r,"w_{\underline{g}}^{p^n}"] & S

\end{tikzcd}
\end{center}

\noindent Let $\iota_n$ be this composition, then $\iota_n(\pi^m\mathcal{O}G)\subseteq \pi^m S$ for all $m$, and $\iota_n(g)\equiv\begin{cases} 
      g &  v(k_{g,\underline{g}})=0 \\
      0 &  v(k_{g,\underline{g}})>0 \\ 
   \end{cases}$
 ($mod$ $\pi^{n+1}S$). Hence $\iota_n(r)\equiv\iota_{n+1}(r)$ ($mod$ $\pi^{n+1}$S) for all $n\in\mathbb{N}$, $r\in\mathcal{O}G$.\\
 
\noindent Therefore, since $S$ is $\pi$-adically complete, there exists a continuous, $\mathcal{O}$-linear morphism $\iota:\mathcal{O}G\to S$ such that $\iota(P)=0$, $\iota(x)\equiv\iota_n(x)$ ($mod$ $\pi^{n+1}S$) for all $n\in\mathbb{N}$, hence for all $g\in G$:

\begin{equation}
\iota(g)=\begin{cases} 
      g &  v(k_{g,\underline{g}})=0 \\
      0 &  v(k_{g,\underline{g}})>0 \\ 
   \end{cases}
\end{equation}

\subsection{The Controlling subgroup}

We will now complete our proof of Theorem \ref{B} by finding an appropriate subgroup of $G$ controlling our ideal $P$.\\

\noindent Let $U:=\{g\in G:v(\tau(z(g)-1))>\lambda\}$, then using the proof of \cite[Lemma 7.6]{nilpotent} we see that $U$ is a proper, open subgroup of $G$ containing $G^p$. So fix an ordered basis $\underline{g}=\{g_1,\cdots,g_d\}$ for $G$ such that $\{g_1^p,\cdots,g_r^p,g_{r+1},\cdots,g_d\}$ is an ordered basis for $U$. We want to prove that $P$ is controlled by $U$.\\

Recall from the previous subsection the definition of $k_{g,\underline{g}}\in\mathcal{V}$ for each $g\in G$:\\

\begin{lemma}\label{subgroup}

$U=\{g\in G:v(k_{g,\underline{g}})>0\}$

\end{lemma}

\begin{proof}

Using Proposition \ref{convergence}, we see that for each $i=1,\cdots,d$, $\beta_i=\underset{j\rightarrow\infty}{\lim}{q_{1,m_j}^{-1}q_{i,m_j}}$, as $m_j$ runs over some subsequence $\underline{a}=(m_1,m_2,\cdots\cdots)$, where $q_{i,m}=\tau(z(g_i)^{p^m}-1)$.\\

\noindent By Lemma \ref{value}, $v(q_{i,m})=\lambda+mv(p)$ for $i\leq r$, and $v(q_{i,m})>\lambda+mv(p)$ for all $i>r$. Hence $v(\beta_i)=0$ for all $i\leq r$, $v(\beta_i)>0$ for all $i>r$.\\

\noindent Given $g\in U$, $g=g_1^{p\alpha_1}\cdots g_r^{p\alpha_r} g_{r+1}^{\alpha_{r+1}}\cdots g_d^{\alpha_d}$ for $\alpha_i\in\mathbb{Z}_p$, so:

\begin{center}
$v(k_{g,\underline{g}})=v(p\alpha_1\beta_1+p\alpha_2\beta_2+\cdots+p\alpha_r\beta_r+\alpha_{r+1}\beta_{r+1}+\cdots+\alpha_d\beta_d)>0$
\end{center}

Conversely, if $v(k_{g,\underline{g}})>0$, suppose $g=\underline{g}^{\alpha}$, so $k_{g,\underline{g}}=\beta_1\alpha_1+\beta_2\alpha_2+\cdots+\beta_d\alpha_d$.\\

\noindent By the definition of $\beta_i$, this means that $v(q_{1,m_j}^{-1}(\alpha_1q_{1,m_j}+\cdots+\alpha_d q_{d,m_j}))>0$ for sufficiently high $j$, and hence $v(\alpha_1q_{1,m_j}+\cdots+\alpha_r q_{d,m_j})>\lambda+m_jv(p)$.\\

But it is easily seen that $\tau(z(g)^{p^{m_j}}-1)\equiv\alpha_1q_{1,m_j}+\cdots+\alpha_dq_{d,m_j}$ ($mod$ $\lambda+m_jv(p)+1$), and hence $v(\tau(z(g)^{p^{m_j}}-1))>\lambda+m_jv(p)$ for sufficiently high $j$.\\

\noindent But $v(\tau(z(g)^{p^{m_j}}-1))=v(\tau(z(g)-1))+m_jv(p)$ by Lemma \ref{value}, and hence $v(\tau(z(g)-1))>\lambda$ and $g\in U$ as required.\end{proof}

\noindent Now, recall the definition of the continuous $\mathcal{O}$-linear morphism $\iota:\mathcal{O}G\to \frac{\mathcal{O}G}{P}$. Then using the lemma, we deduce that $\iota(g)=\begin{cases} 
      g &  g\notin U \\
      0 &  g\in U \\ 
   \end{cases}$.\\
   
Define $f:G\to\mathcal{O},g\mapsto\begin{cases} 
      1 &  g\notin U \\
      0 &  g\in U \\ 
   \end{cases}$, then clearly $f\in C(G,\mathcal{O})$ is locally constant, so the endomorphism $\rho(f)\in End_{\mathcal{O}}(\mathcal{O}G)$ is well defined, and $\rho(f)(g)=f(g)g=\begin{cases} 
      g &  g\notin U \\
      0 &  g\in U \\ 
   \end{cases}$.\\
   
\noindent Therefore $\tau\rho(f)=\iota$ when restricted to $\mathcal{O}[G]$, so since $\iota$ is continuous, $\tau$ is continuous, and $\rho(f)$ is continuous, it follows that $\tau\rho(f)=\iota$. Hence $\tau\rho(f)(P)=0$ and $\rho(f)(P)\subseteq P$.

\begin{proposition}\label{invariant}
$P$ is controlled by $U$.
\end{proposition}

\begin{proof}

Firstly, suppose that $C=\{x_1,\cdots,x_t\}$ is a complete set of coset representatives for $U$ in $G$, then for all $r\in\mathcal{O}G$, $r=\underset{i\leq t}{\sum}{r_ix_i}$ for some $r_i\in\mathcal{O}U$.\\

\noindent Suppose we can choose $C$ such that if $r\in P$ then $r_1\in P\cap\mathcal{O}U$. Then since $rx_1^{-1}x_i\in P$ for all $i=1\cdots,t$ and $rx_1^{-1}x_i$ has $x_1$ component $r_i$, it follows that $r_i\in P\cap\mathcal{O}U$ for each $i$, and hence $P$ is controlled by $U$.

So it remains to prove that we can choose such a set $C$ of coset representatives such that if $\underset{i\leq t}{\sum}{r_ix_i}\in P$, then at least one of the $r_i$ lies in $P\cap\mathcal{O}U$.\\

\noindent Since $U$ has ordered basis $\{g_1^p,\cdots,g_r^p,g_{r+1},\cdots,g_d\}$ it follows that $C=\{g_1^{b_1}\cdots g_r^{b_r}:0\leq b_i<p\}$ is a complete set of coset representatives for $U$ in $G$.

So for each $\underline{b}\in [p-1]^r$, let $g_{\underline{b}}=g_1^{b_1}\cdots g_r^{b_r}$ (here $[p-1]=\{0,1,\cdots,p-1\}$).\\

\noindent Then if $r=\underset{\underline{b}\in [p-1]^r}{\sum}{r_{\underline{b}}g_{\underline{b}}}\in P$, then $\rho(f)(r)=\underset{\underline{b}\in [p-1]^r}{\sum}{f(g_{\underline{b}}})r_{\underline{b}}g_{\underline{b}}$ by the proof of \cite[Proposition 2.5]{nilpotent}, and this also lies in $P$.\\

\noindent But $f(g_{\underline{b}})=1$ if $\underline{b}\neq 0$, and $f(g_{\underline{0}})=0$ hence $\rho(f)(r)=\underset{\underline{b}\in [p-1]^r\backslash\{\underline{0}\}}{\sum}{r_{\underline{b}}g_{\underline{b}}}\in P$.\\

\noindent Therefore, $r_{\underline{0}}g_{\underline{0}}=r-\rho(f)(r)\in P$, and thus $r_{\underline{0}}\in P\cap\mathcal{O}U$ as required.\end{proof}

\noindent The main theorem of this section follows immediately:\\

\noindent\emph{Proof of Theorem \ref{B}.} Since we are assuming that $\mathcal{O}A/\mathcal{O}A\cap P$ is finitely generated over $\mathcal{O}$, it follows from Proposition \ref{invariant} that $P$ is controlled by a proper open subgroup of $G$.\qed

\section{Control Theorem}

Now that we have established Theorem \ref{B}, we will conclude in this section with our main control theorem. Throughout, we will fix $G$ a compact $p$-adic Lie group.

\subsection{Technical results in characteristic 0}

In \cite[Section 5]{nilpotent}, a number of technical results were stated and proved for completed group algebras in characteristic $p$. These results are fundamental for the study of ideals in Iwasawa algbras, so we will first carry them over to a characteristic zero setting.

\begin{lemma}\label{important}

Let $H$ be a closed subgroup of $G$, and let $I_1,\cdots,I_m,J$ be right ideals of $KH$. Then:

($i$) $I_1KG\cap\cdots\cap I_mKG=(I_1\cap\cdots\cap I_m)KG$.

($ii$) $JKG\cap KH=J$.

\end{lemma}

\begin{proof}

We will prove that $KG$ is faithfully flat over $KH$. Then part ($i$) follows from applying the functor $-\otimes_{KH}KG$ to the short exact sequence $0\to I_1\cap\cdots\cap I_m\to KH\to\underset{j\leq m}{\oplus}{\frac{KH}{I_j}}\to 0$, and part ($ii$) follows from \cite[Lemma 7.2.5]{McConnell}, taking $R=KH$, $S=KG$ and $M=\frac{KH}{J}$.\\

\noindent By \cite[Lemma 4.5]{brumer}, $\mathcal{O}G$ is projective over $\mathcal{O}H$, and hence it is flat. It follows immediately that $KG$ is flat over $KH$, so it remains to prove that $KG$ is faithful, i.e. if $M\neq 0$ is a $KH$-module, then $KG\otimes_{KH} M\neq 0$.\\

\noindent Note that $K\cong KH/(H-1)KH$ is a $KH$-module, so if $KG\otimes_{KH}M=0$ then we have the following isomorphisms of $K$-spaces:

\begin{center}
$0=KG\otimes_{KH} M\cong KG\otimes_{KH} K\otimes_K M$
\end{center}

But since $K$ is a field, every module over $K$ is faithfully flat, so if $M\neq 0$, it follows that $KG\otimes_{KH} K=0$, which is impossible since the natural map $KG\otimes_{KH} K\to K,r\otimes\alpha\to r\alpha+(G-1)KG$ is surjective. Hence $KG$ is faithfully flat as required.\end{proof}

Using this Lemma, we can now carry over the proofs of every result in \cite[Chapter 5]{nilpotent}, with the exception of \cite[Lemma 5.3]{nilpotent}, whose proof strongly depends on the assumption that the ground field has characteristic $p$.

\begin{lemma}\label{orbit}

Let $G$ be a $p$-valuable group, let $I$ be a two-sided ideal of $KG$ such that $I=J_1\cap\cdots\cap J_r$ for some right ideals $J_i$ of $KG$ forming a complete $G$-orbit via the conjugation action. Then if $I$ is faithful, each $J_i$ is faithful.

\end{lemma}

\begin{proof}

For any ideal $J$ of $KG$, let $J^{\dagger}:=\{g\in G:g-1\in J\}$. Then clearly $J^{\dagger}$ is a closed subgroup of $G$, and if $J$ is a two-sided ideal then it is a normal subgroup. Clearly $J$ is faithful if and only in $J^{\dagger}=1$.\\

\noindent It is also clear that $(J_1)^{\dagger},\cdots,(J_r)^{\dagger}$ form a single $G$-orbit, and $I^{\dagger}=(J_1)^{\dagger}\cap\cdots\cap (J_r)^{\dagger}$.\\

\noindent But since $G$ is $p$-valuable, it follows from \cite[Proposition 5.9]{nilpotent} that $G$ is \emph{orbitally sound}, i.e. for any closed subgroup $H$ of $G$ with finitely many $G$-conjugates, the intersection of these conjugates is open in $H$. Therefore $I^{\dagger}$ has finite index in $(J_i)^{\dagger}$ for each $i$.

So if $I$ is faithful, then $(J_i)^{\dagger}$ is a finite subgroup of $G$. So since $G$ is torsionfree, this means that $(J_i)^{\dagger}=1$, and $J_i$ is faithful as required.\end{proof}

\noindent Let us now reintroduce some definitions from \cite[Section 5]{nilpotent}:

\begin{definition}\label{p-0def}

$(i)$. Given a prime ideal $P$ of $KG$, we say that $P$ is \emph{non-splitting} if for any open normal subgroup $U$ of $G$ controlling $P$, $P\cap KU$ is prime in $KU$.\\

$(ii)$. Let $\mathcal{P}$ be a property satisfied by two-sided ideals in $KH$, for $H$ any compact $p$-adic Lie group. Then given a right ideal $I$ of $KG$, we say that $I$ \emph{virtually} satisfies $\mathcal{P}$ if there exists an open subgroup $U$ of $G$ and a two sided ideal $J$ of $KU$ such that $J$ satisfies $\mathcal{P}$ and $I=JKG$.\\

\noindent In particular, $I$ is \emph{virtually non-splitting} if $I=PKU$ for some non-splitting prime ideal $P$ of $KU$.

\end{definition}

\begin{proposition}\label{p-01}

Let $P$ be a non-splitting prime ideal of $KG$, then $P\cap KP^{\chi}$ is prime in $KP^{\chi}$

\end{proposition}

\begin{proof}

This is the proof of \cite[Proposition 5.5]{nilpotent}, applied using Lemma \ref{important}.\end{proof}

\noindent Now, recall from \cite[Definition 5.6]{nilpotent} that if $R$ is a ring, $J_1,\cdots,J_r$ are right ideals of $R$ with intersection $I$, then $I=J_1\cap\cdots\cap J_r$ is an \emph{essential decomposition} for $I$ if the $R$-module embedding $\frac{R}{I}\to \frac{R}{J_1}\times\cdots\times \frac{R}{J_r}$ has essential image in the sense of \cite[Definition 2.2.1]{McConnell}.

\begin{proposition}\label{p-02}

Suppose that $G$ is pro-$p$, let $P$ be a prime ideal of $KG$, and let $P=I_1\cap\cdots\cap I_r$ be an essential decomposition for $P$ such that each $I_j$ is virtually prime and $I_1,\cdots I_r$ form a single $G$-orbit. If we assume that $r$ is as large as possible then each $I_j$ is virtually non-splitting.

\end{proposition}

\begin{proof}

This is the proof of \cite[Theorem 5.7]{nilpotent}, applied using Lemma \ref{important}.\end{proof}

\subsection{$J$-Ideals}

Recall that we are interested in \emph{primitive ideals} in $KG$, that is ideals of the form $Ann_{KG}M$ for some irreducible $KG$-module $M$. It is well known that these ideals are prime.

\begin{proposition}\label{rational}

Let $P$ be a primitive ideal of $KG$, then $Z(KG/P)$ is finite dimensional over $K$.

\end{proposition}

\begin{proof}

By definition, $P=Ann_{KG}M$ for some irreducible $KG$-module $M$. Then it follows from \cite[Theorem 1.1(1)]{endomorphism} that the algebra End$_{K[G]}M$ of $G$-equivariant $K$-linear endomorphisms of $M$ is a finite dimensional $K$-algebra.\\

\noindent Now, there exists a well defined, injective homomorphism of $K$-algebras: 

\begin{center}
$Z(KG/P)\to$ End$_{K[G]}M, z+P\mapsto (\phi_z:M\to M, m\mapsto zm)$
\end{center} 

\noindent so it follows immediately that $Z(KG/P)$ is finite dimensional over $K$.\end{proof}

So, we can safely assume that $P$ is a prime ideal of $KG$ such that $Z(KG/P)$ is a finite field extension of $K$, however it will be necessary for us to work in slightly more generality:

\begin{definition}\label{J-ideal}

Let $G$ be a compact $p$-adic Lie group, $Z:=Z(G)$. Given a right ideal $I$ of $KG$, we say that $I$ is a \emph{$J$-ideal} of $KG$ if $KZ(G)/I\cap KZ(G)$ is finite dimensional over $K$.

\end{definition}

Using Proposition \ref{rational}, it is clear that primitive ideals are $J$-ideals. The reason why we prefer to work with ideals of this form is because even when $Z(KG/P)$ is finite dimensional, after passing to a smaller subgroup $H$, there is no reason why $Z(KH/P\cap KH)$ should be finite dimensional. So it makes sense to only consider elements of the full centre $Z(G)$.

\begin{lemma}\label{virtual}

Let $U$ be an open normal subgroup of $G$, $P$ a prime ideal of $KU$ such that $PKU$ is a $J$-ideal of $KG$. Then $P$ is a $J$-ideal of $KU$.

\end{lemma}

\begin{proof}

Since $PKG\cap KU=P$ by Lemma \ref{important}($ii$), we have an injection of $KU$-modules $KU/P\xhookrightarrow{}KG/PKG$, and this map sends $KZ(U)/KZ(U)\cap P$ to $\frac{KZ(U)+KZ(G)\cap PKG}{KZ(G)\cap PKG}$.\\

\noindent But $Z(U)\subseteq Z(G)$ by \cite[Lemma 8.4(b)]{nilpotent}, so this image is contained in $KZ(G)/KZ(G)\cap PKG$, which is finite dimensional over $K$ by the definition of a $J$-ideal. 

Hence $KZ(U)/KZ(U)\cap P$ is finite dimensional over $K$ as required.\end{proof}

\begin{lemma}\label{decomposition}

Suppose $I$ is a $J$-ideal in $KG$ and $I\subseteq J$ for some right ideal $J$ of $KG$. Then $J$ is a $J$-ideal of $KG$.

\end{lemma}

\begin{proof}

Setting $Z:=Z(G)$, since $I$ is a $J$-ideal, $KZ/I\cap KZ$ is finite dimensional over $K$.\\

But $I\subseteq J$, so there is a surjection $KG/I\to KG/J$ of $KG$-modules, and $KZ/KZ\cap J$ is the image of $KZ/KZ\cap I$ under this map.\\

\noindent So since $KZ/KZ\cap J$ is the image of a finite dimensional $K$-vector space under a $K$-linear map, it is also finite dimensional over $K$, making $J$ a $J$-ideal.\end{proof}

\begin{theorem}\label{split}

Let $G$ be a $p$-valuable group, let $A$ be a closed subgroup of $G$, and suppose that all faithful, virtually non-splitting right $J$-ideals of $KG$ are controlled by $A$. Then all faithful, prime $J$-ideals of $KG$ are controlled by $A$.

\end{theorem}

\begin{proof}

This is similar to the proof of \cite[Theorem 5.8]{nilpotent}.\\

\noindent Let $P$ be a faithful, prime $J$-ideal of $KG$, and let $P=I_1\cap\cdots\cap I_m$ be an essential decomposition for $P$, with each $I_j$ virtually prime, and $I_1,\cdots,I_r$ forming a single $G$-orbit.\\ 

\noindent Setting $r=1$, $I_1=P$, it is clear that such a decomposition exists, so we will assume that $r$ is maximal such that a decomposition of this form exists. We know that $r$ is finite because $KG/P$ has finite uniform dimension in the sense of \cite{McConnell}.\\

\noindent So, by Theorem \ref{p-02}, each $I_j$ is a virtually non-splitting right ideal of $KG$, and since $P\subseteq I_j$ it follows from Lemma \ref{decomposition} that $I_j$ is a $J$-ideal. Furthermore, since $P$ is faithful, it follows from Lemma \ref{orbit} that each $I_j$ is faithful.

Therefore, by assumption, $I_j$ is controlled by $A$, so $I_j=(I_j\cap KA)KG$ for each $j$. So since $P=I_1\cap\cdots\cap I_r$, we have that\\

$(P\cap KA)KG=((I_1\cap KA)\cap\cdots\cap (I_r\cap KA))KG$\\

$=(I_1\cap KA)KG\cap\cdots\cap (I_r\cap KA)KG=I_1\cap\cdots\cap I_r=P$ by Lemma \ref{important}($i$).\\ 

\noindent Thus $P$ is controlled by $A$ as required.\end{proof}

\subsection{Controlling prime $J$-ideals}

Recall from section 1 the definition of the second centre $Z_2(G)=\{g\in G:(g,G)\subseteq Z(G)\}$.

\begin{lemma}\label{centraliser}

For $G$ any $p$-valuable group, $U\leq_o G$, $C_U(Z_2(U))=C_G(Z_2(G))\cap U$.

\end{lemma}

\begin{proof}

\noindent Firstly, we will show that $Z_2(U)\subseteq Z_2(G)$:\\ 

\noindent If $g\in Z_2(U)$ then $(g,U)\subseteq Z(U)$ by definition, and $Z(U)\subseteq Z(G)$ by \cite[Lemma 8.4(b)]{nilpotent}, so $(g,U)\subseteq Z(G)$.

But we know that $G^{p^n}\subseteq U$ for some $n\in\mathbb{N}$, so for all $h\in G$, $(g,h^{p^n})\in Z(G)$.\\

\noindent But $Z(G)$ is an isolated normal subgroup of $G$ by \cite[Lemma 8.4(a)]{nilpotent}, so by \cite[\rom{4}.3.4.2]{Lazard}, $\frac{G}{Z(G)}$ carries a $p$-valuation $\bar{\omega}$. So since $(g,h^{p^n})\in Z(G)$, we have that $\bar{\omega}((g,h^{p^n}))=\infty$. 

But $\bar{\omega}((g,h^{p^n}))=\bar{\omega}(g,h)+n$ by \cite[Proposition 25.1]{Schneider}, and hence $\bar{\omega}((g,h))=\infty$. Thus $(g,h)\in Z(G)$ for all $h\in G$, and $g\in Z_2(G)$ as required.\\

\noindent Now, given $g\in C_U(Z_2(U))$, $(g,Z_2(U))=1$, so if $h\in Z_2(G)$ then $h^{p^n}\in Z_2(G)\cap U=Z_2(U)$, so $(g,h^{p^n})=1$. So since $G$ is $p$-valued, applying \cite[Proposition 25.1]{Schneider} again gives that $(g,h)=1$, thus $g\in C_G(Z_2(G))$ and $C_U(Z_2(U))\subseteq C_G(Z_2(G))$.\\

\noindent Conversely, if $g\in C_G(Z_2(G))\cap U$ then $(g,Z_2(G))=1$, so since $Z_2(U)\subseteq Z_2(G)$, $(g,Z_2(U))=1$ and $g\in C_U(Z_2(U))$ as required.\end{proof}

\begin{theorem}\label{J-control}

Let $G$ be a nilpotent, $p$-valuable group, let $P$ be a faithful prime $J$-ideal of $KG$. Then $P$ is controlled by $C_G(Z_2(G))$.

\end{theorem}

\begin{proof}

First, suppose that $P$ is non-splitting, and let $H:=P^{\chi}$ be the controller subgroup of $P$:\\

\noindent Then $Q:=P\cap KH$ is a faithful, prime ideal of $KH$ by Proposition \ref{p-01}, and since $H$ is the smallest subgroup of $G$ controlling $P$, $Q$ is not controlled by any proper subgroup of $H$.

Also, note that $H$ is a normal subgroup of $G$ by the proof of \cite[Lemma 5.2]{nilpotent}, so for any $g\in G$, $(g,H)\subseteq H$, and clearly $A:=Z(G)\cap H\subseteq Z(H)$ is the subgroup of $G$-invariants in $H$.\\

\noindent Since $P$ is a $J$-ideal of $KG$, $KZ(G)/KZ(G)\cap P$ is finite dimensional over $K$ by definition. And:

\begin{center}
$\frac{KA}{KA\cap Q}=\frac{KA}{KA\cap P}\subseteq \frac{KZ(G)}{KZ(G)\cap P}$.
\end{center}

\noindent Therefore, $KA/KA\cap Q$ is finite dimensional over $K$, and hence $\mathcal{O}A/\mathcal{O}A\cap Q$ is finitely generated over $\mathcal{O}$.\\

\noindent So, given $g\in Z_2(G)$, $(g,H)\subseteq Z(G)\cap H=A$, so let $\varphi$ be the automorphism of $H$ induced by conjugation by $g$, then $\varphi(Q)=Q$ and $\varphi(h)h^{-1}\in A$ for all $h\in H$. So applying Theorem \ref{B} gives that if $\varphi\neq 1$, then $Q$ is controlled by a proper subgroup of $H$ -- contradiction.

Therefore $\varphi=1$, i.e. $g$ centralises $H$.\\

\noindent Since our choice of $g$ was arbitrary, we have now proved that $Z_2(G)$ centralises $H$, and hence $H$ is contained in the centraliser $C_G(Z_2(G))$ of $Z_2(G)$ in $G$ as required. Thus $P$ is controlled by $C_G(Z_2(G))$.\\

\noindent Now suppose that $I\trianglelefteq_r KG$ is a faithful and virtually non-splitting right $J$-ideal of $KG$. Then $I=PKG$ for some open subgroup $U$ of $G$, and some faithful, non-splitting prime $P$ of $KU$, and $P$ is a $J$-ideal of $KU$ by Lemma \ref{virtual}.\\

\noindent We have proved that $P$ is controlled by $C_U(Z_2(U))$, and $C_U(Z_2(U))=C_G(Z_2(G))\cap U$ by Lemma \ref{centraliser}, and hence $I$ is controlled by $C_G(Z_2(G))$.

So, using Theorem \ref{split}, it follows that every faithful, prime $J$-ideal of $KG$ is controlled by $C_G(Z_2(G))$.\end{proof}

\noindent We can now prove our main control theorem:\\

\noindent\emph{Proof of Theorem \ref{A}.} Let $P$ be a faithful, primitive ideal of $KG$. Then using Proposition \ref{rational}, we see that $Z(KG/P)$ is finite dimensional over $K$.

Let $Z=Z(G)$, then since $KZ/P\cap KZ\subseteq Z(KG/P)$, it follows that $KZ/P\cap KZ$ is also finite dimensional over $K$, thus $P$ is a $J$-ideal of $KG$.\\

\noindent So since $P$ is a faithful, prime $J$-ideal of $KG$, it follows from Theorem \ref{J-control} that $P$ is controlled by $C_G(Z_2(G))$.\qed

\bibliographystyle{abbrv}

\end{document}